\documentclass[12pt, a4paper, leqno]{amsart}

\usepackage{amsmath}
\usepackage{amsfonts}
\usepackage{amssymb}
\usepackage{times}
\usepackage[T1]{fontenc} 
\usepackage{url}
\usepackage{color,esint}
\usepackage{pstricks-add}
\usepackage{graphicx}
\usepackage[colorlinks]{hyperref}

\let\oldmarginpar\marginpar
\renewcommand\marginpar[1]{\-\oldmarginpar[\raggedleft\footnotesize #1]%
{\raggedright\footnotesize #1}}

\newcommand{\R}{\mathbb{R}}
\newcommand{\Rn}{{\mathbb{R}^d}}

\newcommand{\Rk}{{\mathbb{R}^2}}

\newcommand{\dimH}{\dim_{\mathcal{H}}}
\newcommand{\dimM}{\dim_{\mathcal{M}}}

           \def\XXint#1#2#3{{\setbox0=\hbox{$#1{#2#3}{\int}$}
                \vcenter{\hbox{$#2#3$}}\kern-.5\wd0}}

\def\le{\leqslant}

\def\ge{\geqslant}

\def\phi{\varphi}
\def\rho{\varrho}

\def\dist{\qopname\relax o{dist}}

\usepackage{amsthm}

\theoremstyle{plain}
\newtheorem{thm}[equation]{Theorem}
\newtheorem{lem}[equation]{Lemma}
\newtheorem{prop}[equation]{Proposition}

\theoremstyle{definition}
\newtheorem{defn}[equation]{Definition}

\theoremstyle{remark}

\numberwithin{equation}{section}

\def\be{\begin{equation}}
\def\ee{\end{equation}}

\title{Examples of fractals satisfying the quasihyperbolic boundary condition}
\author{Petteri Harjulehto and Riku Kl\'en }
\date{\today}

\begin{document}

\subjclass[2010]{28A80 (30F45)}
%
\begin{abstract}
In this paper we give explicit examples of bounded domains that satisfy the quasihyperbolic  boundary condition and calculate the values for the constants. These domains are also John domains and we calculate John constants as well. The authors do not know any other paper where exact values of parameters has been estimated.
\end{abstract}

\maketitle

\section{Introduction}

Domains that satisfy the quasihyperbolic boundary condition with a constant $\beta \in (0,1]$ (see  Definition~\ref{def:QHBC}) were introduced Gehring and Martio in \cite{GehM85} and after that they have been studied intensively. The constant $\beta$  plays a crucial role in these studies and many properties have been proved in the terms of it. For example in  \cite{KosR97} Koskela and Rohde showed that the Minkowski dimension of the boundary of the domain is at most $d-c\beta^{d-1}$, where $d$ is the dimension of the boundary of the domain and the constant $c$ depends only on the dimension $d$.  Another example is the paper \cite{HurMV} by Hurri-Syrjänen, Marola and Vähäkangas, where the Poincar\'e inequality is stated in terms of $\beta$.  However, there seems to be very few examples where the exact value for $\beta$ is known. In fact the authors do not know any nontrivial example with  exact constants.

John domains form a proper subclass of domains that satisfy the quasihyperbolic boundary condition \cite[Lemma~3.11]{GehM85}. They were originally introduced in \cite{Joh61} but the more intensive studies started from the article \cite{MarS79} by Martio and Sarvas. John domains are recognized as a wide class of irregular domains where the classical results are known to hold, see for example the article \cite{BucK95} by Buckley and Koskela. Thus it is surprising that the value of the parameter is known only for trivial examples; all proofs seems to give only existence of the parameters. The aim of this paper is to give explicit examples of these domains.

\begin{figure}[ht]
  \includegraphics[width=.35\textwidth]{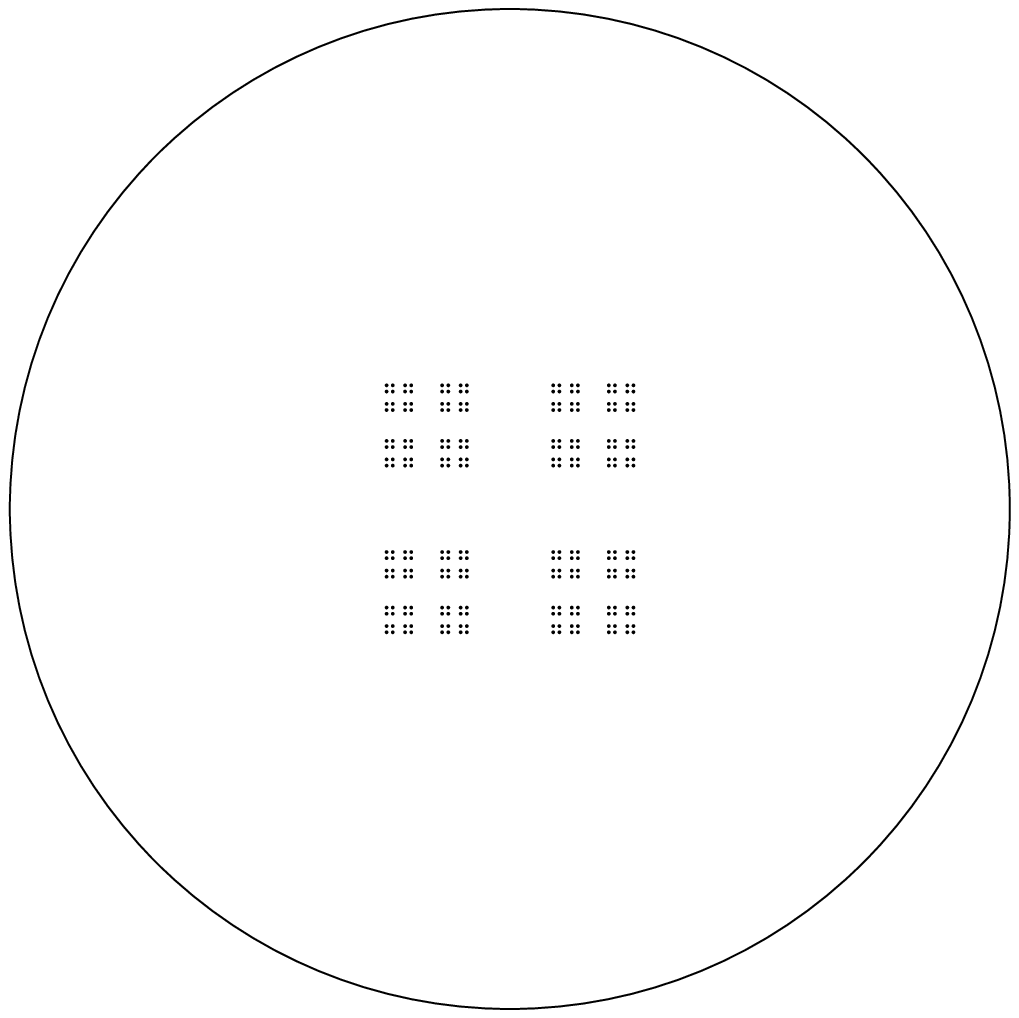}\hspace{1cm}
  \includegraphics[width=.35\textwidth]{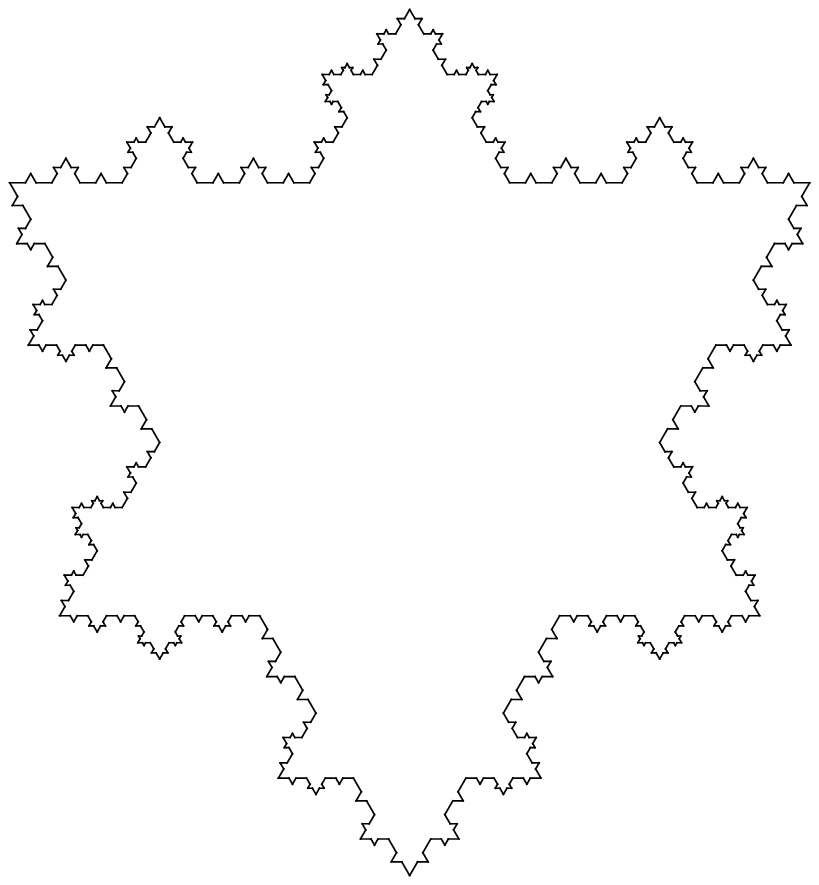}
  \caption{\label{alueet} Left: a Cantor dust-type domain with $\alpha = 1/3$.\newline Right: von Koch snowflake domain with $a = 1/4$.}
\end{figure}

We remove a Cantor dust-type fractal with a ratio $\alpha \in (0,1)$ from an open ball $B(0,2) \subset \Rk$, see Figure~\ref{alueet}.  Then we calculate two constants $\beta_1$ and $\beta_2$ depending only on $\alpha$ and show that our domain  satisfies the quasihyperbolic boundary condition for $\beta\le \beta_1$ and it does not satisfies the quasihyperbolic boundary condition for $\beta\ge \beta_2$ (Theorem~\ref{thm:QHBC}). Although $\beta_1 < \beta_2$, we see that $\beta_2-\beta_1 < 0.04$. Similarly we analyze when this domain is a John domain and show that it is $4.37/\alpha$-John (Theorem~\ref{thm:John}).

We construct a von Koch snowflake in the plane by replacing the middle $a$-th portion, $a\in (0,\tfrac12]$, of each line segment by  the other two sides of an equilateral triangle, see Figure~\ref{alueet}.
We show that the von Koch snowflake domain  satisfies the quasihyperbolic boundary condition  for $\beta\le \beta'_1$ but not for $\beta\ge \beta'_2$ (Theorem~\ref{koch:QHBC}), here $\beta'_1$ and $\beta'_2$ depend only on $a$.  Finally we show that the von Koch snowflake domain is a John domain with a constant $\max\Big\{2, \tfrac{4}{3(1-a)}\Big\}$  (Theorem~\ref{koch:John}). So in particularly for $a\in (0, \tfrac13]$ it is $2$-John.  In this range the result is sharp and  surprisingly the constant does not depend on the parameter $a$ since the worst case is the equilateral triangle inside  the von Koch snowflake domain and every equilateral triangle is a John domain with a constant $2$.

\section{Preliminary results}

Let $D\subsetneq\Rn$ be a domain. The quasihyperbolic length of a rectifiable curve $\gamma \subset D$ is
\[
  \ell_k(\gamma) = \int_\gamma \frac{|dz|}{\dist(z,\partial D)},
\]
where $\dist (z,\partial D)$ is the Euclidean distance between $z$ and $\partial D$.
The quasihyperbolic distance $k_D$ is defined by
\[
  k_{D}(x,y)=\inf_{\gamma} \ell_k (\gamma), \quad x,y\in D,
\]
where the infimum is taken over all rectifiable curves in $D$ joining $x$ and $y$.
By the definition it is clear that the quasihyperbolic metric is monotone with respect to domains, which means that if $D \subsetneq \Rn$ and $D' \subset D$ are domains, and $x,y \in D'$, then $k_D(x,y) \le k_{D'}(x,y)$.

We recall next the definitions of the quasihyperbolic boundary condition and the class of John domains.

\begin{defn}\cite{GehM85}\label{def:QHBC}
A domain $D \subsetneq \Rn$ satisfies a \emph{quasihyperbolic boundary condition} with constants $\beta\in(0,1]$ and $c>0$, or shortly $D$ satisfies $\beta$-QHBC, if there exists a distinguished point  $x_0 \in D$ such that
\begin{equation}\label{QHBC}
  k_D(x_0,x) \le \frac{1}{\beta}\log\frac{1}{\dist(x,\partial D)}+c
\end{equation}
for all $x \in D$.
\end{defn}

Note that if $D' \subset D$ and $x, y \in D'$, then  $k_D(x, y) \le k_{D'}(x,y)$. We use this property when we obtain lower estimates for the quasihyperbolic distance.

\begin{defn}\cite {MarS79}
A domain $D$ is a \emph{$c$-John domain}, $c \ge 1$, if there is a distinguished point $x_0 \in D$ such that any $x \in D$ can be connected to $x_0$ by a rectifiable curve $\gamma:[0, l] \to D$, which is parametrized by arclength and with $\gamma(0)=x$, $\gamma(l)=x_0$ and
\[
\dist(\gamma(t), \partial D) \ge \frac1c t
\]   
for every $0\le t\le l$. The distinguished point $x_0$ is called the John center.
\end{defn}

Punctured space $\Rn \setminus \{ 0 \}$ is one of the very few domains where the explicit formula for the quasihyperbolic distance is known. Martin and Osgood proved the following result in 1986 \cite[p. 38]{MarOsg86}.

\begin{prop}\label{MOprop}
  Let $G = \Rn \setminus \{ 0 \}$ and $x,y \in G$. Then
  \[
    k_G(x,y) = \sqrt{\theta^2+\log^2 \frac{|x|}{|y|}},
  \]
where $\theta = \measuredangle (x,0,y)$.
\end{prop}

Finally, we give a formula for the quasihyperbolic length of a Euclidean line segment in twice-punctured space.

\begin{lem}{\cite[Remark 4.26]{Kle09}}\label{QHjana}
  Let $G = \Rn \setminus \{ a,b \}$ for $a \neq b$, $c = (a+b)/2$, the line $l$ be the perpendicular bisector of $[a,b]$ and $x \in l$. Then
  \[
    \ell_k([x,c]) = \log \left(2 \left( \left| x-c \right| + \sqrt{  |a-b|^2/4  +  \left| x-c \right| ^2 } \right) \right) - \log |a-b|.
  \]
\end{lem}

\section{Cantor dust-type fractal}

Let $\alpha \in (0,1)$.
Let $Q_0\subset \Rk$ be the closed square in the plane which side length is 1 and which is centered at the origin. We make a Cantor construction in $Q_0$. We remove from $Q_0$ strips $\{-\tfrac{\alpha}2 < x < \tfrac{\alpha}2\}$ and $\{-\tfrac{\alpha}2 < y < \tfrac{\alpha}2\}$. We get four closed squares $Q_1^j$, $j=1, \ldots, 2^{2}$. We continue the process by removing from each $Q_1^j$ vertical and horizontal strips of width $\alpha \time \ell(Q_i^j)$. We set
\[
C_\alpha = \bigcap_{i=1}^\infty \bigcup_{j=1}^{2^{2i}} Q_i^j \cap Q_0.
\]
Thus $C_\alpha$ consists of the corner points of all squares $Q_i^j$.
The set $C_\alpha$ is self-similar and thus its Hausdorff dimension is equal to its Minkowski dimension \cite[Lemma 3.1, p. 488]{Lap91}. By \cite[Theorem~9.3, p.~118]{Fal90} we can calculate
\[
\dimH (C_\alpha) = \dimM (C_\alpha) = \frac{\log4}{\log\tfrac{2}{1-\alpha}}.
\]
Thus $\alpha \mapsto  \dimH (C_\alpha) = \dimM (C_\alpha)$ is a strictly decreasing bijective mapping from $(0,1)$  to $(0,2)$. Note that in the range $\alpha \in (0, \tfrac12]$ we have $\dimH (C_\alpha) = \dimM (C_\alpha) <1$.  

We set
\[
\Omega_\alpha = B(0, 2) \setminus C_\alpha \subset \Rk.
\]
Then $\Omega_\alpha$ is a bounded domain with $\dimH (\partial \Omega_\alpha) = \dimM (\partial\Omega_\alpha) = \max\big\{1, \dimM (C_\alpha)\big\}$ and for every $\lambda\in [1,2)$ there exists a unique $\alpha \in [0,\tfrac12]$ such that $\lambda = \dimH (\partial \Omega_\alpha) = \dimM (\partial\Omega_\alpha)$. 
For the domain $\Omega_\alpha $ see Figure \ref{alueet} or \ref{pic:geodeesi1}.

\begin{thm}\label{thm:QHBC}
The domain $\Omega_\alpha \subset \Rk$ (defined above) satisfies the $\beta$-QHBC for
\begin{equation}\label{equ:upperbound}
  \beta \le \beta_1 = \frac{\log \frac{2}{1-\alpha}}{\log{\frac{2 + \sqrt{4 + (1-\alpha)^2}}{1-\alpha}}  + \frac3{2\alpha} + \frac{\pi}2- \frac32}
\end{equation}
and it does not satisfy $\beta$-QHBC for
\begin{equation}\label{equ:lowerbound}
  \beta \ge \beta_2 = \frac{\log \frac{2}{1-\alpha}}{\log{\frac{2 + \sqrt{4 + (1-\alpha)^2}}{1-\alpha}}  + \frac{1-\alpha}{\alpha} +  \frac{\pi}{2}}.
\end{equation}
\end{thm}

\begin{figure}[ht]
  \includegraphics[width=.45\textwidth]{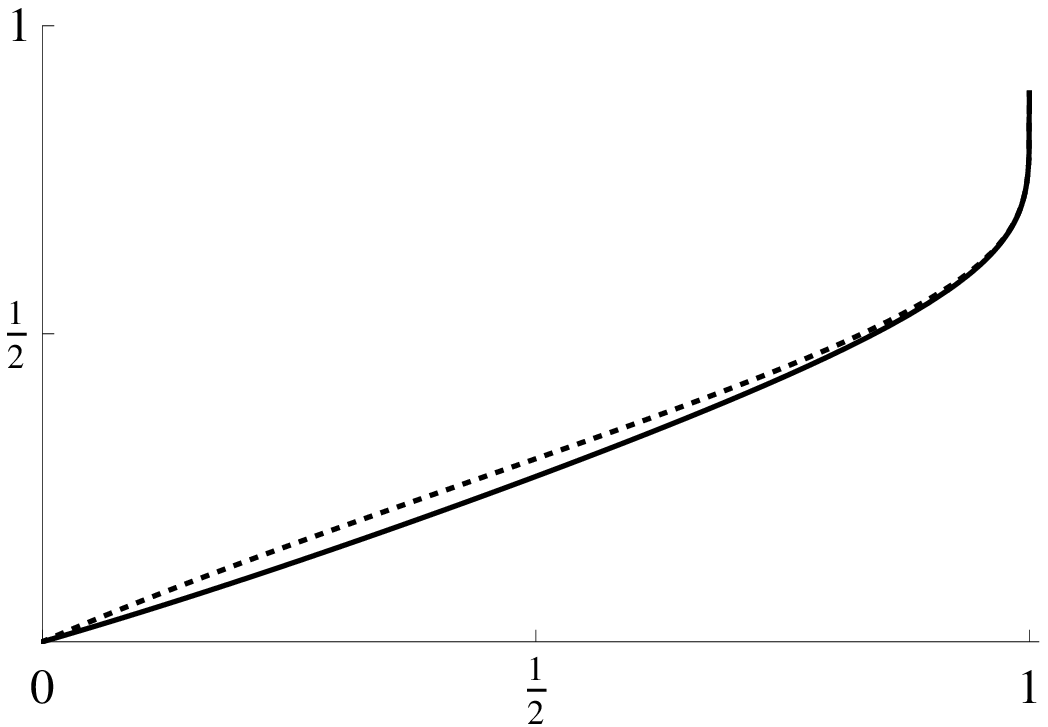}\hspace{1cm}
  \includegraphics[width=.45\textwidth]{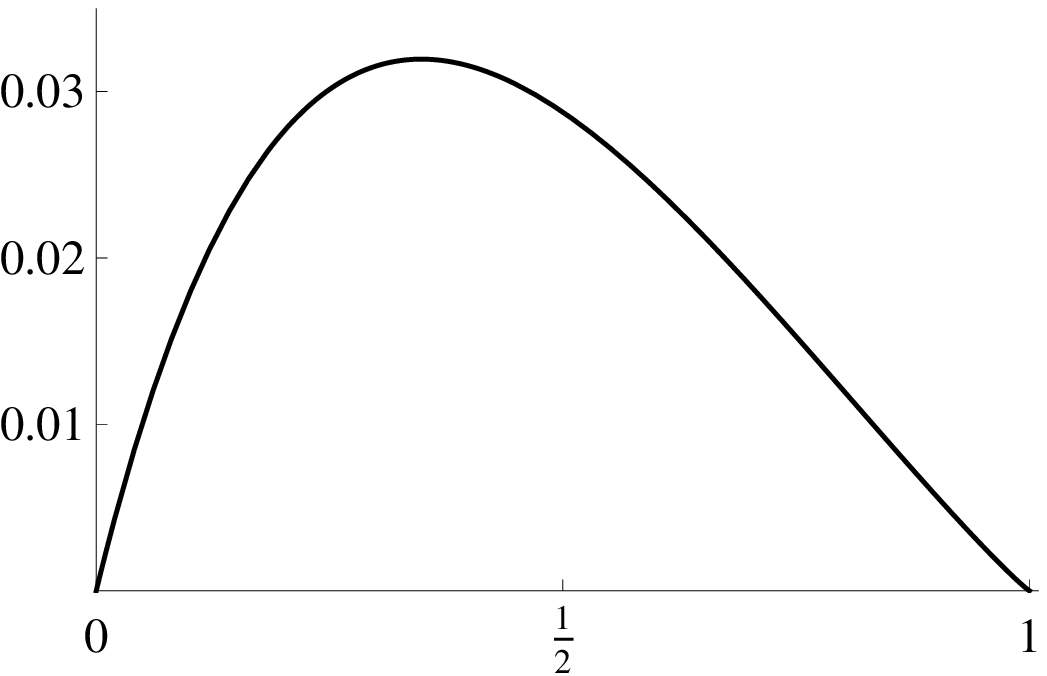}
  \caption{\label{CDbetas} Left: bounds $\beta_1$ (solid line) and $\beta_2$ (dashed line) of Theorem \ref{thm:QHBC} plotted as functions of $\alpha$. Right:$\beta_2 - \beta_1$ plotted as a function of $\alpha$.}
\end{figure}

Note that although $\beta_1 < \beta_2$ we have $\beta_2-\beta_1 < 0.04$, see Figure \ref{CDbetas}.  

\begin{proof}
Let $x_0$ be a center of $Q_0$ and let $x_n$ be a center of $Q_n^j$ in an upper right corner, see Figure~\ref{pic:geodeesi1}. We want to give upper and lower estimates for the quasihyperbolic distance $k_{\Omega_\alpha}(x_0, x_n)$.  Then by the geometry of the domain we can connect by a line segment any $x\in Q_0\cap \Omega_\alpha$ to  a suitable center point. Thus if the center points satisfy the $\beta$-QHBC then, by increasing the constant $c$ in \eqref{QHBC}, all $x\in Q_0\cap \Omega_\alpha$ satisfy $\beta$-QHBC for the same $\beta$.    
We start with the upper estimate. We connect $x_0$ and $x_n$ as in  the Figure \ref{pic:geodeesi1}, where we use line segments and circle arcs near the points $x_1$,\dots,$x_{n-1}$. Let us denote $l \in \{ 1,2,\dots ,n \}$. We first estimate the dotted part of the path denoted by $p_l$.
Let $y_l$ and $u_l$ be as in Figure~\ref{pic:geodeesi1}.
By Lemma \ref{QHjana} we obtain
\[
\begin{split}
  k(p_l) & = k(y_l,u_l) \le k([y_l,u_l])\\
  & = \log\Big(\alpha \big(\tfrac{1-\alpha}{2}\big)^{l-1} \Big( 1 + \sqrt{1+ \tfrac14 (1-\alpha)^2} \Big)\Big)- \log\big(\alpha \big(\tfrac{1-\alpha}{2}\big)^{l}\big)\\
 &= \log{\frac{2 + \sqrt{4 + (1-\alpha)^2}}{1-\alpha}}.
\end{split}
\]

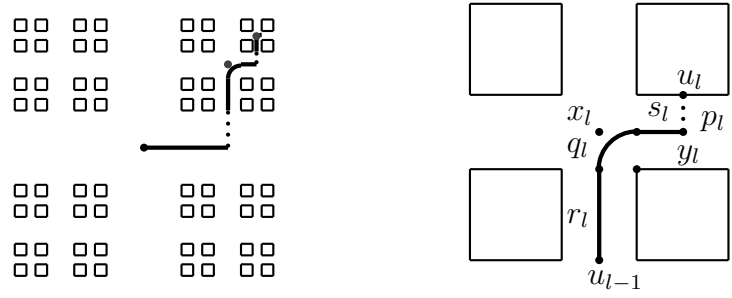
\begin{figure}[ht]
\psset{xunit=17mm,yunit=17mm,algebraic=true,dotstyle=o,dotsize=3pt 0,linewidth=0.8pt,arrowsize=3pt 2,arrowinset=0.25}
\begin{pspicture*}(-1.63,-1.2)(2.19,1.2)
\psline(0.91,1)(1,1)
\psline(1,1)(1,0.91)
\psline(1,0.91)(0.91,0.91)
\psline(0.91,0.91)(0.91,1)
\psline(0.91,0.84)(1,0.84)
\psline(1,0.84)(1,0.75)
\psline(1,0.75)(0.91,0.75)
\psline(0.91,0.75)(0.91,0.84)
\psline(0.91,0.54)(1,0.54)
\psline(1,0.54)(1,0.45)
\psline(1,0.45)(0.91,0.45)
\psline(0.91,0.45)(0.91,0.54)
\psline(0.91,0.38)(1,0.38)
\psline(1,0.38)(1,0.29)
\psline(1,0.29)(0.91,0.29)
\psline(0.91,0.29)(0.91,0.38)
\psline(0.84,0.38)(0.75,0.38)
\psline(0.75,0.38)(0.75,0.29)
\psline(0.75,0.29)(0.84,0.29)
\psline(0.84,0.29)(0.84,0.38)
\psline(0.84,0.45)(0.75,0.45)
\psline(0.75,0.45)(0.75,0.54)
\psline(0.75,0.54)(0.84,0.54)
\psline(0.84,0.54)(0.84,0.45)
\psline(0.84,0.75)(0.75,0.75)
\psline(0.75,0.75)(0.75,0.84)
\psline(0.75,0.84)(0.84,0.84)
\psline(0.84,0.84)(0.84,0.75)
\psline(0.75,0.91)(0.75,1)
\psline(0.75,1)(0.84,1)
\psline(0.84,1)(0.84,0.91)
\psline(0.84,0.91)(0.75,0.91)
\psline(0.29,1)(0.38,1)
\psline(0.38,1)(0.38,0.91)
\psline(0.38,0.91)(0.29,0.91)
\psline(0.29,0.91)(0.29,1)
\psline(0.45,1)(0.54,1)
\psline(0.54,1)(0.54,0.91)
\psline(0.54,0.91)(0.45,0.91)
\psline(0.45,0.91)(0.45,1)
\psline(0.29,0.84)(0.38,0.84)
\psline(0.38,0.84)(0.38,0.75)
\psline(0.38,0.75)(0.29,0.75)
\psline(0.29,0.75)(0.29,0.84)
\psline(0.45,0.84)(0.54,0.84)
\psline(0.54,0.84)(0.54,0.75)
\psline(0.54,0.75)(0.45,0.75)
\psline(0.45,0.75)(0.45,0.84)
\psline(0.29,0.54)(0.38,0.54)
\psline(0.38,0.54)(0.38,0.45)
\psline(0.38,0.45)(0.29,0.45)
\psline(0.29,0.45)(0.29,0.54)
\psline(0.45,0.54)(0.54,0.54)
\psline(0.54,0.54)(0.54,0.45)
\psline(0.54,0.45)(0.45,0.45)
\psline(0.45,0.45)(0.45,0.54)
\psline(0.45,0.38)(0.54,0.38)
\psline(0.54,0.38)(0.54,0.29)
\psline(0.54,0.29)(0.45,0.29)
\psline(0.45,0.29)(0.45,0.38)
\psline(0.29,0.38)(0.38,0.38)
\psline(0.38,0.38)(0.38,0.29)
\psline(0.38,0.29)(0.29,0.29)
\psline(0.29,0.29)(0.29,0.38)
\psline(0.91,-1)(1,-1)
\psline(1,-1)(1,-0.91)
\psline(1,-0.91)(0.91,-0.91)
\psline(0.91,-0.91)(0.91,-1)
\psline(0.91,-0.84)(1,-0.84)
\psline(1,-0.84)(1,-0.75)
\psline(1,-0.75)(0.91,-0.75)
\psline(0.91,-0.75)(0.91,-0.84)
\psline(0.91,-0.54)(1,-0.54)
\psline(1,-0.54)(1,-0.45)
\psline(1,-0.45)(0.91,-0.45)
\psline(0.91,-0.45)(0.91,-0.54)
\psline(0.91,-0.38)(1,-0.38)
\psline(1,-0.38)(1,-0.29)
\psline(1,-0.29)(0.91,-0.29)
\psline(0.91,-0.29)(0.91,-0.38)
\psline(0.84,-0.38)(0.75,-0.38)
\psline(0.75,-0.38)(0.75,-0.29)
\psline(0.75,-0.29)(0.84,-0.29)
\psline(0.84,-0.29)(0.84,-0.38)
\psline(0.84,-0.45)(0.75,-0.45)
\psline(0.75,-0.45)(0.75,-0.54)
\psline(0.75,-0.54)(0.84,-0.54)
\psline(0.84,-0.54)(0.84,-0.45)
\psline(0.84,-0.75)(0.75,-0.75)
\psline(0.75,-0.75)(0.75,-0.84)
\psline(0.75,-0.84)(0.84,-0.84)
\psline(0.84,-0.84)(0.84,-0.75)
\psline(0.75,-0.91)(0.75,-1)
\psline(0.75,-1)(0.84,-1)
\psline(0.84,-1)(0.84,-0.91)
\psline(0.84,-0.91)(0.75,-0.91)
\psline(0.29,-1)(0.38,-1)
\psline(0.38,-1)(0.38,-0.91)
\psline(0.38,-0.91)(0.29,-0.91)
\psline(0.29,-0.91)(0.29,-1)
\psline(0.45,-1)(0.54,-1)
\psline(0.54,-1)(0.54,-0.91)
\psline(0.54,-0.91)(0.45,-0.91)
\psline(0.45,-0.91)(0.45,-1)
\psline(0.29,-0.84)(0.38,-0.84)
\psline(0.38,-0.84)(0.38,-0.75)
\psline(0.38,-0.75)(0.29,-0.75)
\psline(0.29,-0.75)(0.29,-0.84)
\psline(0.45,-0.84)(0.54,-0.84)
\psline(0.54,-0.84)(0.54,-0.75)
\psline(0.54,-0.75)(0.45,-0.75)
\psline(0.45,-0.75)(0.45,-0.84)
\psline(0.29,-0.54)(0.38,-0.54)
\psline(0.38,-0.54)(0.38,-0.45)
\psline(0.38,-0.45)(0.29,-0.45)
\psline(0.29,-0.45)(0.29,-0.54)
\psline(0.45,-0.54)(0.54,-0.54)
\psline(0.54,-0.54)(0.54,-0.45)
\psline(0.54,-0.45)(0.45,-0.45)
\psline(0.45,-0.45)(0.45,-0.54)
\psline(0.45,-0.38)(0.54,-0.38)
\psline(0.54,-0.38)(0.54,-0.29)
\psline(0.54,-0.29)(0.45,-0.29)
\psline(0.45,-0.29)(0.45,-0.38)
\psline(0.29,-0.38)(0.38,-0.38)
\psline(0.38,-0.38)(0.38,-0.29)
\psline(0.38,-0.29)(0.29,-0.29)
\psline(0.29,-0.29)(0.29,-0.38)
\psline(-0.91,1)(-1,1)
\psline(-1,1)(-1,0.91)
\psline(-1,0.91)(-0.91,0.91)
\psline(-0.91,0.91)(-0.91,1)
\psline(-0.91,0.84)(-1,0.84)
\psline(-1,0.84)(-1,0.75)
\psline(-1,0.75)(-0.91,0.75)
\psline(-0.91,0.75)(-0.91,0.84)
\psline(-0.91,0.54)(-1,0.54)
\psline(-1,0.54)(-1,0.45)
\psline(-1,0.45)(-0.91,0.45)
\psline(-0.91,0.45)(-0.91,0.54)
\psline(-0.91,0.38)(-1,0.38)
\psline(-1,0.38)(-1,0.29)
\psline(-1,0.29)(-0.91,0.29)
\psline(-0.91,0.29)(-0.91,0.38)
\psline(-0.84,0.38)(-0.75,0.38)
\psline(-0.75,0.38)(-0.75,0.29)
\psline(-0.75,0.29)(-0.84,0.29)
\psline(-0.84,0.29)(-0.84,0.38)
\psline(-0.84,0.45)(-0.75,0.45)
\psline(-0.75,0.45)(-0.75,0.54)
\psline(-0.75,0.54)(-0.84,0.54)
\psline(-0.84,0.54)(-0.84,0.45)
\psline(-0.84,0.75)(-0.75,0.75)
\psline(-0.75,0.75)(-0.75,0.84)
\psline(-0.75,0.84)(-0.84,0.84)
\psline(-0.84,0.84)(-0.84,0.75)
\psline(-0.75,0.91)(-0.75,1)
\psline(-0.75,1)(-0.84,1)
\psline(-0.84,1)(-0.84,0.91)
\psline(-0.84,0.91)(-0.75,0.91)
\psline(-0.29,1)(-0.38,1)
\psline(-0.38,1)(-0.38,0.91)
\psline(-0.38,0.91)(-0.29,0.91)
\psline(-0.29,0.91)(-0.29,1)
\psline(-0.45,1)(-0.54,1)
\psline(-0.54,1)(-0.54,0.91)
\psline(-0.54,0.91)(-0.45,0.91)
\psline(-0.45,0.91)(-0.45,1)
\psline(-0.29,0.84)(-0.38,0.84)
\psline(-0.38,0.84)(-0.38,0.75)
\psline(-0.38,0.75)(-0.29,0.75)
\psline(-0.29,0.75)(-0.29,0.84)
\psline(-0.45,0.84)(-0.54,0.84)
\psline(-0.54,0.84)(-0.54,0.75)
\psline(-0.54,0.75)(-0.45,0.75)
\psline(-0.45,0.75)(-0.45,0.84)
\psline(-0.29,0.54)(-0.38,0.54)
\psline(-0.38,0.54)(-0.38,0.45)
\psline(-0.38,0.45)(-0.29,0.45)
\psline(-0.29,0.45)(-0.29,0.54)
\psline(-0.45,0.54)(-0.54,0.54)
\psline(-0.54,0.54)(-0.54,0.45)
\psline(-0.54,0.45)(-0.45,0.45)
\psline(-0.45,0.45)(-0.45,0.54)
\psline(-0.45,0.38)(-0.54,0.38)
\psline(-0.54,0.38)(-0.54,0.29)
\psline(-0.54,0.29)(-0.45,0.29)
\psline(-0.45,0.29)(-0.45,0.38)
\psline(-0.29,0.38)(-0.38,0.38)
\psline(-0.38,0.38)(-0.38,0.29)
\psline(-0.38,0.29)(-0.29,0.29)
\psline(-0.29,0.29)(-0.29,0.38)
\psline(-0.91,-1)(-1,-1)
\psline(-1,-1)(-1,-0.91)
\psline(-1,-0.91)(-0.91,-0.91)
\psline(-0.91,-0.91)(-0.91,-1)
\psline(-0.91,-0.84)(-1,-0.84)
\psline(-1,-0.84)(-1,-0.75)
\psline(-1,-0.75)(-0.91,-0.75)
\psline(-0.91,-0.75)(-0.91,-0.84)
\psline(-0.91,-0.54)(-1,-0.54)
\psline(-1,-0.54)(-1,-0.45)
\psline(-1,-0.45)(-0.91,-0.45)
\psline(-0.91,-0.45)(-0.91,-0.54)
\psline(-0.91,-0.38)(-1,-0.38)
\psline(-1,-0.38)(-1,-0.29)
\psline(-1,-0.29)(-0.91,-0.29)
\psline(-0.91,-0.29)(-0.91,-0.38)
\psline(-0.84,-0.38)(-0.75,-0.38)
\psline(-0.75,-0.38)(-0.75,-0.29)
\psline(-0.75,-0.29)(-0.84,-0.29)
\psline(-0.84,-0.29)(-0.84,-0.38)
\psline(-0.84,-0.45)(-0.75,-0.45)
\psline(-0.75,-0.45)(-0.75,-0.54)
\psline(-0.75,-0.54)(-0.84,-0.54)
\psline(-0.84,-0.54)(-0.84,-0.45)
\psline(-0.84,-0.75)(-0.75,-0.75)
\psline(-0.75,-0.75)(-0.75,-0.84)
\psline(-0.75,-0.84)(-0.84,-0.84)
\psline(-0.84,-0.84)(-0.84,-0.75)
\psline(-0.75,-0.91)(-0.75,-1)
\psline(-0.75,-1)(-0.84,-1)
\psline(-0.84,-1)(-0.84,-0.91)
\psline(-0.84,-0.91)(-0.75,-0.91)
\psline(-0.29,-1)(-0.38,-1)
\psline(-0.38,-1)(-0.38,-0.91)
\psline(-0.38,-0.91)(-0.29,-0.91)
\psline(-0.29,-0.91)(-0.29,-1)
\psline(-0.45,-1)(-0.54,-1)
\psline(-0.54,-1)(-0.54,-0.91)
\psline(-0.54,-0.91)(-0.45,-0.91)
\psline(-0.45,-0.91)(-0.45,-1)
\psline(-0.29,-0.84)(-0.38,-0.84)
\psline(-0.38,-0.84)(-0.38,-0.75)
\psline(-0.38,-0.75)(-0.29,-0.75)
\psline(-0.29,-0.75)(-0.29,-0.84)
\psline(-0.45,-0.84)(-0.54,-0.84)
\psline(-0.54,-0.84)(-0.54,-0.75)
\psline(-0.54,-0.75)(-0.45,-0.75)
\psline(-0.45,-0.75)(-0.45,-0.84)
\psline(-0.29,-0.54)(-0.38,-0.54)
\psline(-0.38,-0.54)(-0.38,-0.45)
\psline(-0.38,-0.45)(-0.29,-0.45)
\psline(-0.29,-0.45)(-0.29,-0.54)
\psline(-0.45,-0.54)(-0.54,-0.54)
\psline(-0.54,-0.54)(-0.54,-0.45)
\psline(-0.54,-0.45)(-0.45,-0.45)
\psline(-0.45,-0.45)(-0.45,-0.54)
\psline(-0.45,-0.38)(-0.54,-0.38)
\psline(-0.54,-0.38)(-0.54,-0.29)
\psline(-0.54,-0.29)(-0.45,-0.29)
\psline(-0.45,-0.29)(-0.45,-0.38)
\psline(-0.29,-0.38)(-0.38,-0.38)
\psline(-0.38,-0.38)(-0.38,-0.29)
\psline(-0.38,-0.29)(-0.29,-0.29)
\psline(-0.29,-0.29)(-0.29,-0.38)
\parametricplot[linewidth=1.6pt]{1.5707963267948954}{3.1415926535897922}{1*0.1*cos(t)+0*0.1*sin(t)+0.75|0*0.1*cos(t)+1*0.1*sin(t)+0.54}
\psline[linewidth=1.6pt](0.65,0.54)(0.65,0.29)
\psline[linewidth=1.6pt,linestyle=dotted](0.65,0.29)(0.65,0)
\psline[linewidth=1.6pt](0.65,0)(0,0)
\psline[linewidth=1.6pt](0.75,0.65)(0.87,0.65)
\psline[linewidth=1.6pt,linestyle=dotted](0.87,0.65)(0.87,0.75)
\psline[linewidth=1.6pt](0.87,0.75)(0.87,0.84)
\parametricplot[linewidth=1.6pt]{1.5707963267949272}{3.14159265358979}{1*0.04*cos(t)+0*0.04*sin(t)+0.91|0*0.04*cos(t)+1*0.04*sin(t)+0.84}
\begin{scriptsize}
\psdots[dotstyle=*](0,0)
\psdots[dotstyle=*,linecolor=darkgray](0.65,0.65)
\psdots[dotstyle=*,linecolor=darkgray](0.87,0.87)
\end{scriptsize}
\end{pspicture*}
\begin{pspicture*}(-1.27,-1.32)(1.55,1.1)
\psline(-1,1)(-0.29,1)
\psline(-0.29,1)(-0.29,0.29)
\psline(-0.29,0.29)(-1,0.29)
\psline(-1,0.29)(-1,1)
\psline(0.29,1)(1,1)
\psline(1,1)(1,0.29)
\psline(1,0.29)(0.29,0.29)
\psline(0.29,0.29)(0.29,1)
\psline(-1,-0.29)(-0.29,-0.29)
\psline(-0.29,-0.29)(-0.29,-1)
\psline(-0.29,-1)(-1,-1)
\psline(-1,-1)(-1,-0.29)
\psline(0.29,-0.29)(1,-0.29)
\psline(1,-0.29)(1,-1)
\psline(1,-1)(0.29,-1)
\psline(0.29,-1)(0.29,-0.29)
\psline[linewidth=1.6pt](0,-1)(0,-0.29)
\psline[linewidth=1.6pt](0.29,0)(0.65,0)
\psline[linewidth=1.6pt,linestyle=dotted](0.65,0)(0.65,0.29)
\psarc[linewidth=1.6pt](0.29,-0.29){0.5}{90}{180}
\rput[tl](-0.25,0.2){$ x_l $}
\rput[tl](0.79,0.16){$ p_l $}
\rput[tl](-0.25,-0.05){$ q_{l} $}
\rput[tl](-0.25,-0.6){$ r_{l} $}
\rput[tl](-0.09,-1.08){$ u_{l-1} $}
\rput[tl](0.37,0.23){$s_l $}
\rput[tl](0.6,0.48){$ u_l $}
\rput[tl](0.6,-0.1){$ y_l $}
\begin{scriptsize}
\psdots[dotstyle=*](0,-1)
\psdots[dotstyle=*](0,0)
\psdots[dotstyle=*](0.29,0)
\psdots[dotstyle=*](0.29,-0.29)
\psdots[dotstyle=*](0,-0.29)
\psdots[dotstyle=*](0.65,0)
\psdots[dotstyle=*](0.65,0.29)
\end{scriptsize}
\end{pspicture*}

\caption{The path used in the proof of Theorem \ref{thm:QHBC}. \label{pic:geodeesi1}}

\end{figure}

For the circle arc the radius is $\alpha \tfrac12 \big(\tfrac{1-\alpha}{2}\big)^{l}$ and hence the quasihyperbolic length of the circle arc is
\[
k(q_l) = \frac{\tfrac{\pi}2 \alpha \tfrac12 \big(\tfrac{1-\alpha}{2}\big)^{l}}{\alpha \tfrac12 \big(\tfrac{1-\alpha}{2}\big)^{l}} = \frac{\pi}2.
\]
There are two line segments inside the square $Q_{l-1}^j$. The longer has length $\ell(Q_l^j) = \big(\tfrac{1-\alpha}{2}\big)^{l}$ and the shorter $\tfrac12 \ell(Q_l^j)$. In both parts the distance to the boundary is equal to or greater than $\tfrac12 \alpha \big(\tfrac{1-\alpha}{2}\big)^{l-1}$. For the line segments we obtain an upper bound for the quasihyperbolic length
\[
  k(r_l) + k(s_l) = \frac{\tfrac32 \big(\tfrac{1-\alpha}{2}\big)^{l}}{\tfrac12 \alpha \big(\tfrac{1-\alpha}{2}\big)^{l-1}} = \frac3{\alpha} \Big(\frac{1-\alpha}{2}\Big).
\]
Putting these three estimates together and adding the first and last parts of the path, we have 

\begin{eqnarray}
k_{\Omega_\alpha}(x_0, x_n) & \le &  \frac{\tfrac12 \alpha + \frac12 \big(\tfrac{1-\alpha}{2}\big) }{\tfrac12 \alpha} + n  \log{\frac{2 + \sqrt{4 + (1-\alpha)^2}}{1-\alpha}}\nonumber\\
  & & + (n-1)\bigg(\frac3{\alpha} \Big(\frac{1-\alpha}{2}\Big) +  \frac{\pi}2 \bigg) + \frac{\tfrac12 \big(\tfrac{1-\alpha}{2}\big)^n}{\tfrac12 \alpha \big(\tfrac{1-\alpha}{2}\big)^n}\nonumber\\
&=& 2  - \frac{\pi}2 + n \Big( \log{\frac{2 + \sqrt{4 + (1-\alpha)^2}}{1-\alpha}}  + \frac3{2\alpha} + \frac{\pi}2- \frac32\Big).\label{kupperbound}
\end{eqnarray}

Next we calculate a lower bound for the quasihyperbolic distance. We do not need to know where exactly quasihyperbolic geodesic is located. But if a geodesic connects $x_0$ and $x_n$ in the upper right corner, then the geodesic should go from the boundary of $Q_{l}^j$ to the boundary  of $Q_{l+1}^j$. Thus we can give a lower estimate to the quasihyperbolic distance $k(u_{l-1},u_l)$. First we estimate the path from the  boundary of $Q_{l}^j$ to the 'middle square' of $Q_{l}^j$. Here the shortest route is in the middle of the strip and in the same time the distance to $C_\alpha$ is the greatest. Thus we obtain
\[
k(r_l) \ge \frac{\big(\tfrac{1-\alpha}{2}\big)^{l+1}}{\frac12 \alpha \big(\tfrac{1-\alpha}{2}\big)^{l}} = \frac{2}{\alpha} \Big(\frac{1-\alpha}{2}\Big). 
\]
Then we estimate the path across the 'middle square' to the boundary of $Q_{l+1}^j$.
We use a circular arc to estimate the path through the 'middle square', see Figure \ref{pic:geodeesi1}, and obtain $k(q_l) \ge \tfrac{\pi}{2}$. Finally we estimate the path from the 'middle square' to the boundary of $Q_{l+1}^j$. In the boundary of $Q_{l+1}^j$ the distance to $C_\alpha$ is at most $\tfrac12 \alpha \ell(Q_{l+1}^j)$. Thus we get a lower estimate for the later half by approaching to the middle of the strip perpendicular to the boundary of $Q_{l+1}^j$ as we did in the dotted part of the upper bound. We get the term $\log{\frac{2 + \sqrt{4 + (1-\alpha)^2}}{1-\alpha}}$. 
Collecting the terms together we obtain
\begin{eqnarray}
k_{\Omega_\alpha}(x_0,x_n) &\ge n \frac{2}{\alpha} \Big(\frac{1-\alpha}{2}\Big) + (n-1) \frac{\pi}{2} + n \log{\frac{2 + \sqrt{4 + (1-\alpha)^2}}{1-\alpha}}\nonumber\\
& = -\frac{\pi}{2} + n \left( \log{\frac{2 + \sqrt{4 + (1-\alpha)^2}}{1-\alpha}}  + \frac{1-\alpha}{\alpha} +  \frac{\pi}{2} \right).\label{klowerbound}
\end{eqnarray}

In the definition of the quasihyperbolic boundary condition we choose $x_0 = 0$ and let $x = x_n$ be a center of $Q_n^j$. Now
\[
  \dist(x_n,\partial C_\alpha) = \sqrt{2} \frac{\alpha}{2} \left( \frac{1-\alpha}{2} \right)^n
\]
and thus
\begin{equation}\label{distx0}
  \log\frac{1}{\dist(x_n,\partial C_\alpha)} = \log \frac{\sqrt{2}}{\alpha}+n \log \frac{2}{1-\alpha}.
\end{equation}
Combining \eqref{kupperbound} and \eqref{distx0}   and letting $n \to \infty$ we deduce that $\Omega_\alpha$ satisfies  \eqref{QHBC} in the QHBC for
\[
  \beta \le \frac{\log \frac{2}{1-\alpha}}{\log{\frac{2 + \sqrt{4 + (1-\alpha)^2}}{1-\alpha}}  + \frac3{2\alpha} + \frac{\pi}2- \frac32}.
\]
Similarly combining \eqref{klowerbound} and \eqref{distx0} and letting $n \to \infty$ we see that $\Omega_\alpha$ does not satisfies \eqref{QHBC} in the definition of the QHBC for
\[
  \beta \ge \frac{\log \frac{2}{1-\alpha}}{\log{\frac{2 + \sqrt{4 + (1-\alpha)^2}}{1-\alpha}}  + \frac{1-\alpha}{\alpha} +  \frac{\pi}{2}}. \qedhere
\]
\end{proof}

\begin{prop}\label{prop:John}
Let $0$ be the John center. Then 
the domain $\Omega_\alpha$ is $c$-John for $c\ge 4.37/\alpha$, and it is not $c$-John for $c\le 4/\alpha$.
\end{prop}
\begin{proof}
  We consider first the case that $x \in \Omega_\alpha \cap Q_0$. Let $x_n$ be a center of $Q_n^j$ in an upper right corner. We choose the curve $\gamma_{n,0}$ joining $x_n$ and $x_0$ consisting of horizontal and vertical line segments as in Figure \ref{pic:geodeesi2}. We denote $u_k = \gamma_{n,0} \cap \partial Q_k^j$ and $y_k,z_k \in \gamma_{n,0}$ as in Figure \ref{pic:geodeesi2}. Now
  \[
    \ell(\gamma_{n,0}) = 1-\left( \frac{1-\alpha}{2} \right) ^n
  \]
  and
  \[
    \ell(\gamma_{n,k}) = \left( \frac{1-\alpha}{2} \right) ^k - \left( \frac{1-\alpha}{2} \right) ^n,
  \]
  where $\gamma_{n,k}$ is the subcurve of $\gamma_{n,0}$ connecting $x_n$ to $x_k$ with $k < n$. Let $x \in Q_n^j \setminus \cup Q_{n+1}^j$ and $\gamma_{y} = [x,x_n] \cup \gamma_{n,y}$ for $y \in \gamma_{n,0}$, where $\gamma_{n,y}$ is the subcurve of $\gamma_{n,0}$ connecting $x_n$ to $y$. Now $|x-x_n| < ((1-\alpha)/2)^n \sqrt{1+\alpha^2}/2$ implying
  \begin{eqnarray*}
    \frac{\dist\big(\gamma_{z_k}(\ell(\gamma_{z_k}), \partial \Omega_\alpha)\big)}{\ell(\gamma_{z_k})} & \ge & \frac{\frac{\alpha}{2}\left( \frac{1-\alpha}{2} \right) ^{k}}{ \left( \frac{1-\alpha}{2} \right)^n \frac{\sqrt{1+\alpha^2}}{2} + \left( \frac{1-\alpha}{2} \right) ^{k} - \left( \frac{1-\alpha}{2} \right) ^n + \frac{\alpha}{2}\left( \frac{1-\alpha}{2} \right) ^{k}}\\
  & = & \frac{\alpha}{ \left( \frac{1-\alpha}{2} \right) ^{n-k}(\sqrt{1+\alpha^2}-2)+2-\alpha} > \frac{\alpha}{3}
  \end{eqnarray*}
  and
  \begin{eqnarray*}
    \frac{\dist\big(\gamma_{u_k}(\ell(\gamma_{u_k}), \partial \Omega_\alpha)\big)}{\ell(\gamma_{u_k})} & \ge & \frac{\frac{\alpha}{2}\left( \frac{1-\alpha}{2} \right) ^{k+1}}{ \left( \frac{1-\alpha}{2} \right)^n \frac{\sqrt{1+\alpha^2}}{2} + \left( \frac{1-\alpha}{2} \right) ^{k+1} - \left( \frac{1-\alpha}{2} \right) ^n + \frac{1}{2}\left( \frac{1-\alpha}{2} \right) ^{k+1}}\\
  & = & \frac{\alpha}{ \left( \frac{1-\alpha}{2} \right) ^{n-k-1}(\sqrt{1+\alpha^2}-2)+3} > \frac{\alpha}{3}.
  \end{eqnarray*}
Hence the definition holds if $\tfrac1c \le \tfrac{\alpha}{3}$ i.e. if  $c \ge 3/\alpha$.

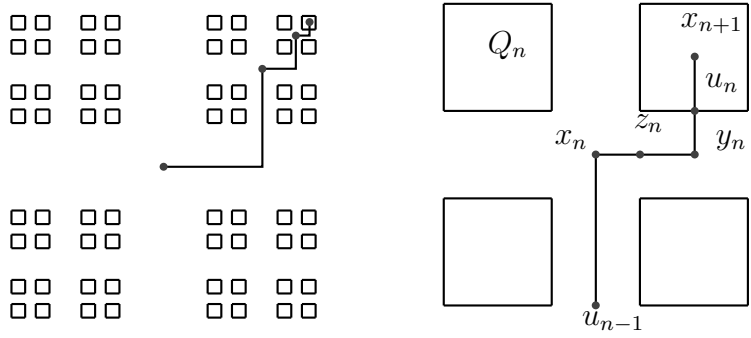
\begin{figure}[ht]
\psset{xunit=20mm,yunit=20mm,algebraic=true,dotstyle=o,dotsize=3pt 0,linewidth=0.8pt,arrowsize=3pt 2,arrowinset=0.25}
\begin{pspicture*}(-1.29,-1.24)(1.52,1.25)
\psline(0.91,1)(1,1)
\psline(1,1)(1,0.91)
\psline(1,0.91)(0.91,0.91)
\psline(0.91,0.91)(0.91,1)
\psline(0.91,0.84)(1,0.84)
\psline(1,0.84)(1,0.75)
\psline(1,0.75)(0.91,0.75)
\psline(0.91,0.75)(0.91,0.84)
\psline(0.91,0.54)(1,0.54)
\psline(1,0.54)(1,0.45)
\psline(1,0.45)(0.91,0.45)
\psline(0.91,0.45)(0.91,0.54)
\psline(0.91,0.38)(1,0.38)
\psline(1,0.38)(1,0.29)
\psline(1,0.29)(0.91,0.29)
\psline(0.91,0.29)(0.91,0.38)
\psline(0.84,0.38)(0.75,0.38)
\psline(0.75,0.38)(0.75,0.29)
\psline(0.75,0.29)(0.84,0.29)
\psline(0.84,0.29)(0.84,0.38)
\psline(0.84,0.45)(0.75,0.45)
\psline(0.75,0.45)(0.75,0.54)
\psline(0.75,0.54)(0.84,0.54)
\psline(0.84,0.54)(0.84,0.45)
\psline(0.84,0.75)(0.75,0.75)
\psline(0.75,0.75)(0.75,0.84)
\psline(0.75,0.84)(0.84,0.84)
\psline(0.84,0.84)(0.84,0.75)
\psline(0.75,0.91)(0.75,1)
\psline(0.75,1)(0.84,1)
\psline(0.84,1)(0.84,0.91)
\psline(0.84,0.91)(0.75,0.91)
\psline(0.29,1)(0.38,1)
\psline(0.38,1)(0.38,0.91)
\psline(0.38,0.91)(0.29,0.91)
\psline(0.29,0.91)(0.29,1)
\psline(0.45,1)(0.54,1)
\psline(0.54,1)(0.54,0.91)
\psline(0.54,0.91)(0.45,0.91)
\psline(0.45,0.91)(0.45,1)
\psline(0.29,0.84)(0.38,0.84)
\psline(0.38,0.84)(0.38,0.75)
\psline(0.38,0.75)(0.29,0.75)
\psline(0.29,0.75)(0.29,0.84)
\psline(0.45,0.84)(0.54,0.84)
\psline(0.54,0.84)(0.54,0.75)
\psline(0.54,0.75)(0.45,0.75)
\psline(0.45,0.75)(0.45,0.84)
\psline(0.29,0.54)(0.38,0.54)
\psline(0.38,0.54)(0.38,0.45)
\psline(0.38,0.45)(0.29,0.45)
\psline(0.29,0.45)(0.29,0.54)
\psline(0.45,0.54)(0.54,0.54)
\psline(0.54,0.54)(0.54,0.45)
\psline(0.54,0.45)(0.45,0.45)
\psline(0.45,0.45)(0.45,0.54)
\psline(0.45,0.38)(0.54,0.38)
\psline(0.54,0.38)(0.54,0.29)
\psline(0.54,0.29)(0.45,0.29)
\psline(0.45,0.29)(0.45,0.38)
\psline(0.29,0.38)(0.38,0.38)
\psline(0.38,0.38)(0.38,0.29)
\psline(0.38,0.29)(0.29,0.29)
\psline(0.29,0.29)(0.29,0.38)
\psline(0.91,-1)(1,-1)
\psline(1,-1)(1,-0.91)
\psline(1,-0.91)(0.91,-0.91)
\psline(0.91,-0.91)(0.91,-1)
\psline(0.91,-0.84)(1,-0.84)
\psline(1,-0.84)(1,-0.75)
\psline(1,-0.75)(0.91,-0.75)
\psline(0.91,-0.75)(0.91,-0.84)
\psline(0.91,-0.54)(1,-0.54)
\psline(1,-0.54)(1,-0.45)
\psline(1,-0.45)(0.91,-0.45)
\psline(0.91,-0.45)(0.91,-0.54)
\psline(0.91,-0.38)(1,-0.38)
\psline(1,-0.38)(1,-0.29)
\psline(1,-0.29)(0.91,-0.29)
\psline(0.91,-0.29)(0.91,-0.38)
\psline(0.84,-0.38)(0.75,-0.38)
\psline(0.75,-0.38)(0.75,-0.29)
\psline(0.75,-0.29)(0.84,-0.29)
\psline(0.84,-0.29)(0.84,-0.38)
\psline(0.84,-0.45)(0.75,-0.45)
\psline(0.75,-0.45)(0.75,-0.54)
\psline(0.75,-0.54)(0.84,-0.54)
\psline(0.84,-0.54)(0.84,-0.45)
\psline(0.84,-0.75)(0.75,-0.75)
\psline(0.75,-0.75)(0.75,-0.84)
\psline(0.75,-0.84)(0.84,-0.84)
\psline(0.84,-0.84)(0.84,-0.75)
\psline(0.75,-0.91)(0.75,-1)
\psline(0.75,-1)(0.84,-1)
\psline(0.84,-1)(0.84,-0.91)
\psline(0.84,-0.91)(0.75,-0.91)
\psline(0.29,-1)(0.38,-1)
\psline(0.38,-1)(0.38,-0.91)
\psline(0.38,-0.91)(0.29,-0.91)
\psline(0.29,-0.91)(0.29,-1)
\psline(0.45,-1)(0.54,-1)
\psline(0.54,-1)(0.54,-0.91)
\psline(0.54,-0.91)(0.45,-0.91)
\psline(0.45,-0.91)(0.45,-1)
\psline(0.29,-0.84)(0.38,-0.84)
\psline(0.38,-0.84)(0.38,-0.75)
\psline(0.38,-0.75)(0.29,-0.75)
\psline(0.29,-0.75)(0.29,-0.84)
\psline(0.45,-0.84)(0.54,-0.84)
\psline(0.54,-0.84)(0.54,-0.75)
\psline(0.54,-0.75)(0.45,-0.75)
\psline(0.45,-0.75)(0.45,-0.84)
\psline(0.29,-0.54)(0.38,-0.54)
\psline(0.38,-0.54)(0.38,-0.45)
\psline(0.38,-0.45)(0.29,-0.45)
\psline(0.29,-0.45)(0.29,-0.54)
\psline(0.45,-0.54)(0.54,-0.54)
\psline(0.54,-0.54)(0.54,-0.45)
\psline(0.54,-0.45)(0.45,-0.45)
\psline(0.45,-0.45)(0.45,-0.54)
\psline(0.45,-0.38)(0.54,-0.38)
\psline(0.54,-0.38)(0.54,-0.29)
\psline(0.54,-0.29)(0.45,-0.29)
\psline(0.45,-0.29)(0.45,-0.38)
\psline(0.29,-0.38)(0.38,-0.38)
\psline(0.38,-0.38)(0.38,-0.29)
\psline(0.38,-0.29)(0.29,-0.29)
\psline(0.29,-0.29)(0.29,-0.38)
\psline(-0.91,1)(-1,1)
\psline(-1,1)(-1,0.91)
\psline(-1,0.91)(-0.91,0.91)
\psline(-0.91,0.91)(-0.91,1)
\psline(-0.91,0.84)(-1,0.84)
\psline(-1,0.84)(-1,0.75)
\psline(-1,0.75)(-0.91,0.75)
\psline(-0.91,0.75)(-0.91,0.84)
\psline(-0.91,0.54)(-1,0.54)
\psline(-1,0.54)(-1,0.45)
\psline(-1,0.45)(-0.91,0.45)
\psline(-0.91,0.45)(-0.91,0.54)
\psline(-0.91,0.38)(-1,0.38)
\psline(-1,0.38)(-1,0.29)
\psline(-1,0.29)(-0.91,0.29)
\psline(-0.91,0.29)(-0.91,0.38)
\psline(-0.84,0.38)(-0.75,0.38)
\psline(-0.75,0.38)(-0.75,0.29)
\psline(-0.75,0.29)(-0.84,0.29)
\psline(-0.84,0.29)(-0.84,0.38)
\psline(-0.84,0.45)(-0.75,0.45)
\psline(-0.75,0.45)(-0.75,0.54)
\psline(-0.75,0.54)(-0.84,0.54)
\psline(-0.84,0.54)(-0.84,0.45)
\psline(-0.84,0.75)(-0.75,0.75)
\psline(-0.75,0.75)(-0.75,0.84)
\psline(-0.75,0.84)(-0.84,0.84)
\psline(-0.84,0.84)(-0.84,0.75)
\psline(-0.75,0.91)(-0.75,1)
\psline(-0.75,1)(-0.84,1)
\psline(-0.84,1)(-0.84,0.91)
\psline(-0.84,0.91)(-0.75,0.91)
\psline(-0.29,1)(-0.38,1)
\psline(-0.38,1)(-0.38,0.91)
\psline(-0.38,0.91)(-0.29,0.91)
\psline(-0.29,0.91)(-0.29,1)
\psline(-0.45,1)(-0.54,1)
\psline(-0.54,1)(-0.54,0.91)
\psline(-0.54,0.91)(-0.45,0.91)
\psline(-0.45,0.91)(-0.45,1)
\psline(-0.29,0.84)(-0.38,0.84)
\psline(-0.38,0.84)(-0.38,0.75)
\psline(-0.38,0.75)(-0.29,0.75)
\psline(-0.29,0.75)(-0.29,0.84)
\psline(-0.45,0.84)(-0.54,0.84)
\psline(-0.54,0.84)(-0.54,0.75)
\psline(-0.54,0.75)(-0.45,0.75)
\psline(-0.45,0.75)(-0.45,0.84)
\psline(-0.29,0.54)(-0.38,0.54)
\psline(-0.38,0.54)(-0.38,0.45)
\psline(-0.38,0.45)(-0.29,0.45)
\psline(-0.29,0.45)(-0.29,0.54)
\psline(-0.45,0.54)(-0.54,0.54)
\psline(-0.54,0.54)(-0.54,0.45)
\psline(-0.54,0.45)(-0.45,0.45)
\psline(-0.45,0.45)(-0.45,0.54)
\psline(-0.45,0.38)(-0.54,0.38)
\psline(-0.54,0.38)(-0.54,0.29)
\psline(-0.54,0.29)(-0.45,0.29)
\psline(-0.45,0.29)(-0.45,0.38)
\psline(-0.29,0.38)(-0.38,0.38)
\psline(-0.38,0.38)(-0.38,0.29)
\psline(-0.38,0.29)(-0.29,0.29)
\psline(-0.29,0.29)(-0.29,0.38)
\psline(-0.91,-1)(-1,-1)
\psline(-1,-1)(-1,-0.91)
\psline(-1,-0.91)(-0.91,-0.91)
\psline(-0.91,-0.91)(-0.91,-1)
\psline(-0.91,-0.84)(-1,-0.84)
\psline(-1,-0.84)(-1,-0.75)
\psline(-1,-0.75)(-0.91,-0.75)
\psline(-0.91,-0.75)(-0.91,-0.84)
\psline(-0.91,-0.54)(-1,-0.54)
\psline(-1,-0.54)(-1,-0.45)
\psline(-1,-0.45)(-0.91,-0.45)
\psline(-0.91,-0.45)(-0.91,-0.54)
\psline(-0.91,-0.38)(-1,-0.38)
\psline(-1,-0.38)(-1,-0.29)
\psline(-1,-0.29)(-0.91,-0.29)
\psline(-0.91,-0.29)(-0.91,-0.38)
\psline(-0.84,-0.38)(-0.75,-0.38)
\psline(-0.75,-0.38)(-0.75,-0.29)
\psline(-0.75,-0.29)(-0.84,-0.29)
\psline(-0.84,-0.29)(-0.84,-0.38)
\psline(-0.84,-0.45)(-0.75,-0.45)
\psline(-0.75,-0.45)(-0.75,-0.54)
\psline(-0.75,-0.54)(-0.84,-0.54)
\psline(-0.84,-0.54)(-0.84,-0.45)
\psline(-0.84,-0.75)(-0.75,-0.75)
\psline(-0.75,-0.75)(-0.75,-0.84)
\psline(-0.75,-0.84)(-0.84,-0.84)
\psline(-0.84,-0.84)(-0.84,-0.75)
\psline(-0.75,-0.91)(-0.75,-1)
\psline(-0.75,-1)(-0.84,-1)
\psline(-0.84,-1)(-0.84,-0.91)
\psline(-0.84,-0.91)(-0.75,-0.91)
\psline(-0.29,-1)(-0.38,-1)
\psline(-0.38,-1)(-0.38,-0.91)
\psline(-0.38,-0.91)(-0.29,-0.91)
\psline(-0.29,-0.91)(-0.29,-1)
\psline(-0.45,-1)(-0.54,-1)
\psline(-0.54,-1)(-0.54,-0.91)
\psline(-0.54,-0.91)(-0.45,-0.91)
\psline(-0.45,-0.91)(-0.45,-1)
\psline(-0.29,-0.84)(-0.38,-0.84)
\psline(-0.38,-0.84)(-0.38,-0.75)
\psline(-0.38,-0.75)(-0.29,-0.75)
\psline(-0.29,-0.75)(-0.29,-0.84)
\psline(-0.45,-0.84)(-0.54,-0.84)
\psline(-0.54,-0.84)(-0.54,-0.75)
\psline(-0.54,-0.75)(-0.45,-0.75)
\psline(-0.45,-0.75)(-0.45,-0.84)
\psline(-0.29,-0.54)(-0.38,-0.54)
\psline(-0.38,-0.54)(-0.38,-0.45)
\psline(-0.38,-0.45)(-0.29,-0.45)
\psline(-0.29,-0.45)(-0.29,-0.54)
\psline(-0.45,-0.54)(-0.54,-0.54)
\psline(-0.54,-0.54)(-0.54,-0.45)
\psline(-0.54,-0.45)(-0.45,-0.45)
\psline(-0.45,-0.45)(-0.45,-0.54)
\psline(-0.45,-0.38)(-0.54,-0.38)
\psline(-0.54,-0.38)(-0.54,-0.29)
\psline(-0.54,-0.29)(-0.45,-0.29)
\psline(-0.45,-0.29)(-0.45,-0.38)
\psline(-0.29,-0.38)(-0.38,-0.38)
\psline(-0.38,-0.38)(-0.38,-0.29)
\psline(-0.38,-0.29)(-0.29,-0.29)
\psline(-0.29,-0.29)(-0.29,-0.38)
\psline(0,0)(0.65,0)(0.65,0.65)(0.87,0.65)(0.87,0.87)(0.96,0.87)(0.96,0.96)
\begin{scriptsize}
\psdots[dotstyle=*,linecolor=darkgray](0,0)
\psdots[dotstyle=*,linecolor=darkgray](0.65,0.65)
\psdots[dotstyle=*,linecolor=darkgray](0.87,0.87)
\psdots[dotstyle=*,linecolor=darkgray](0.96,0.96)
\end{scriptsize}
\end{pspicture*}
\begin{pspicture*}(-1.27,-1.32)(1.55,1.27)
\psline(-1,1)(-0.29,1)
\psline(-0.29,1)(-0.29,0.29)
\psline(-0.29,0.29)(-1,0.29)
\psline(-1,0.29)(-1,1)
\psline(0.29,1)(1,1)
\psline(1,1)(1,0.29)
\psline(1,0.29)(0.29,0.29)
\psline(0.29,0.29)(0.29,1)
\psline(-1,-0.29)(-0.29,-0.29)
\psline(-0.29,-0.29)(-0.29,-1)
\psline(-0.29,-1)(-1,-1)
\psline(-1,-1)(-1,-0.29)
\psline(0.29,-0.29)(1,-0.29)
\psline(1,-0.29)(1,-1)
\psline(1,-1)(0.29,-1)
\psline(0.29,-1)(0.29,-0.29)
\psline(0,-1)(0,0)(0.65,0)(0.65,0.65)
\rput[tl](-0.27,0.16){$ x_n $}
\rput[tl](0.79,0.16){$ y_n $}
\rput[tl](0.56,0.94){$ x_{n+1} $}
\rput[tl](-0.09,-1.04){$ u_{n-1} $}
\rput[tl](0.25,0.27){$ z_n $}
\rput[tl](0.72,0.54){$ u_n $}
\rput[tl](-0.71,0.82){$ Q_n $}
\begin{scriptsize}
\psdots[dotstyle=*,linecolor=darkgray](0,-1)
\psdots[dotstyle=*,linecolor=darkgray](0,0)
\psdots[dotstyle=*,linecolor=darkgray](0.29,0)
\psdots[dotstyle=*,linecolor=darkgray](0.65,0)
\psdots[dotstyle=*,linecolor=darkgray](0.65,0.65)
\psdots[dotstyle=*,linecolor=darkgray](0.65,0.29)
\end{scriptsize}
\end{pspicture*}
\caption{The curve $\gamma_{n,0}$ and points $y_n$, $z_n$ and $u_n$ used in the proof of Proposition~\ref{prop:John}.\label{pic:geodeesi2}}
\end{figure}

  Let us then consider $x_n \in \Omega_\alpha \setminus Q_0 = B^2(2) \setminus Q_0$. Let $x_n = i(2-1/n)$ and $\gamma_n$ is the line segment joining $x_n$ to $x_0$. Now
  \[
    \frac{\dist(\gamma_n(t), \partial \Omega_\alpha)}{\ell(\gamma_n(t))} \le \frac{\frac{\alpha}{2}}{2-1/n} = \frac{\alpha}{4-2/n} < \frac{\alpha}{4}
  \]
  and hence the definition does not hold if $\tfrac1c \ge \tfrac{\alpha}{4}$ i.e. if  $c \le 4/\alpha$.

  Let $x_n = \sqrt{2}(1+i)(2-1/n)$ and $\gamma_n = [x_0,i/2] \cup \delta \cup [\alpha/2+i(1+\alpha)/2,x_n]$, where $\delta$ is the circular arc joining $i/2$ and $\alpha/2+i(1+\alpha)/2$ with center at $\alpha/2+i/2$. Now
  \begin{eqnarray*}
    \frac{\dist(\gamma_n(t), \partial \Omega_\alpha)}{\ell(\gamma_n(t))} & \ge & \lim_{n \to \infty} \frac{\dist(\gamma_n(t))}{\ell(\gamma_n(t))} \\
    & = & \frac{\frac{\alpha}{2}}{\frac{1-\alpha}{2}+\frac{\pi \alpha}{4}+\sqrt{(\sqrt{2}-\frac{1}{2}-\frac{\alpha}{2})^2+(\sqrt{2}-\frac{1}{2}+\frac{1-\alpha}{2})^2}}\\
    & = & \frac{\alpha}{1-\alpha+\frac{\pi \alpha}{2}+\sqrt{17-4\sqrt{2}+2\alpha (1-4\sqrt{2}+\alpha^2)}}\\
    & > & \frac{\alpha}{1+\sqrt{17-4\sqrt{2}}} > \frac{\alpha}{4.37}.
  \end{eqnarray*}
  Hence the definition holds if $\tfrac1c \le \tfrac{\alpha}{4.37}$ i.e. if  $c \ge 4.37/\alpha$.

  By the geometry it is clear that the assertion follows.
\end{proof}

When the parameter $\alpha$ is small then the origin is no longer a good choice for the John center. In the next theorem we use $5i/4$ instead and get a slightly better result. Most probably the optimal John center should depend on $\alpha$ and thus have the form $c(\alpha)i$.  

\begin{thm}\label{thm:John}
The domain $\Omega_\alpha$ is $4.37/\alpha$-John for $\alpha \in [1/3,1)$ and $3/\alpha$-John for $\alpha \in (0,1/3)$.
\end{thm}
\begin{proof}
  By Proposition~\ref{prop:John} the domain $\Omega_\alpha$ is $4.37/\alpha$-John and thus we need to show that for $\alpha < 1/3$ it is $3/\alpha$-John.

  Let $\alpha < 1/3$ and choose $x_0 = 5i/4$ to be the John center. By the proof of Proposition~\ref{prop:John} it is clear that for all $y \in \Omega_\alpha \cap Q_0$ we have
  \[
    \frac{\dist(\gamma(t), \partial \Omega_\alpha)}{\ell(\gamma(t))} > \frac{\alpha}{3},
  \]
  where $\gamma = \gamma' \cup [0,x_0]$ is the curve joining $y$ to $x_0$ and $\gamma$ is as in Figure \ref{pic:geodeesi2}.

  Let us now assume that $y \in B(0,2) \setminus Q_0$. We consider the curve $\gamma$ from $y$ to $x_0$, which consists of the line segment $[y,5y/(4|y|)]$ and the shortest circular arc from $5y/(4|y|)$ to $x_0$ with center at 0. By the selection of $\gamma$ we obtain
  \[
    \frac{\dist(\gamma(t), \partial \Omega_\alpha)}{\ell(\gamma(t))} > \frac{\frac{5}{4}-\frac{\sqrt{2}}{2}}{\pi\frac{5}{4}+\frac{3}{4}} = \frac{5-2\sqrt{2}}{5\pi+3} > \frac{1}{9} > \frac{\alpha}{3},
  \]
  where the last inequality follows from the fact that $\alpha < 1/3$. Now the assertion follows as $\Omega_\alpha$ is $\frac{3}{\alpha}$-John.
\end{proof}

\section{von Koch snowflake domain}

We construct a von Koch snowflake. Let $a \in (0,1/2]$. We start with an equilateral triangle with side length 1. We replace the middle $a$-th portion of each line segment by  the other two sides of an equilateral triangle.  We continue inductively and obtain a von Koch snowflake. We denote by $S_a$ the bounded domain bordered by the von Koch snowflake. Then $\partial S_a$ is self-similar and thus its Hausdorff dimension is equal to its Minkowski dimension \cite[Lemma 3.1, p. 488]{Lap91}. Note that for $a \in (0,1/2)$, $\partial S_a$ is not self-intersecting \cite[Theorem 3.1]{KelP10}. The Minkowski dimension  of $\partial S_a$ is the solution of $2a^s + 2 \big(\tfrac12 (1-a) \big)^s =1$ for $a \in (0,1/2)$,  \cite[Example~9.5, p.~120]{Fal90}.

\begin{thm}\label{koch:QHBC}
  The domain $S_a \subset \R^2$ satisfies the $\beta$-QHBC for
  \[
    \beta \le \beta_1= \frac{\log \frac{1}{a}}{\log \frac{1+\sqrt{1+3a^2}}{a\sqrt{3}} + \log\big(3 +2 \sqrt{3}\big)}
  \]
  and it does not satisfy $\beta$-QHBC for
  \[
    \beta \ge \beta_2= \frac{\log \frac{1}{a}}{\sqrt{\arcsin^2 \frac{\sqrt{3}}{\sqrt{2(1+a)(3+2a)}}+\log^2 \frac{\sqrt{(1+a)(3+2a)}}{a\sqrt{2}}}}.
  \]
\end{thm} 

\begin{figure}[ht]
  \includegraphics[width=.45\textwidth]{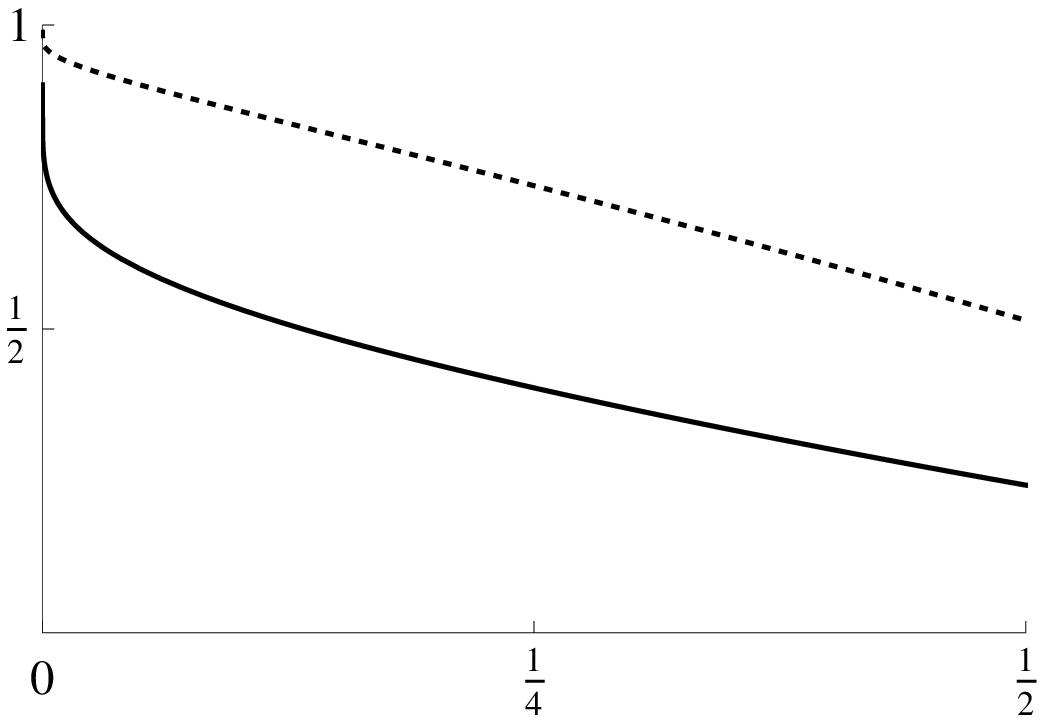}\hspace{1cm}
  \includegraphics[width=.45\textwidth]{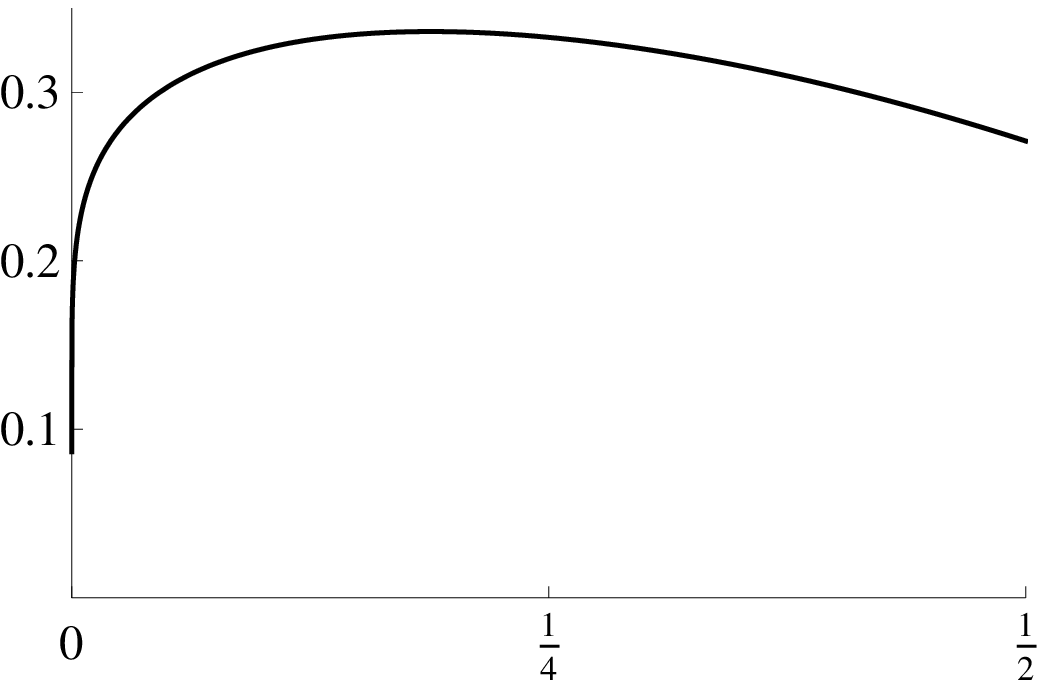}
  \caption{\label{Kbetas} Left: bounds $\beta_1$ (solid line) and $\beta_2$ (dashed line) of Theorem \ref{koch:QHBC} plotted as functions of $\alpha$. Right:$\beta_2 - \beta_1$ plotted as a function of $\alpha$.}
\end{figure}

We have that $\beta_2- \beta_1 < 0.4$, see Figure \ref{Kbetas}.

\begin{proof}
  We calculate first the upper bound $\beta_1$. We concentrate on the worst situation, see Figure~\ref{pic:geodeesi3}, where we first go up and then always to the left to the center of a triangle. Note that other points in the same triangle can be easily connect to the center point and thus they do not effect to the value of $\beta$. Let us denote by $x_0$ the center of $S_a$ and by $x_n$ the center of the triangle constructed on the $n$-th iteration as in Figure \ref{pic:geodeesi3}. We estimate $k_{S_a}(x_0,x_n)$ by using the  curve $\gamma_n = \cup_{i=1}^n [x_{i-1},x_i]$ and denote points $y_n,z_n$ as in Figure \ref{pic:geodeesi3}. We estimate $k_{S_a}(x_n,y_n)$ by the quasihyperbolic length of line segments $[x_n,y_n]$  in the domain $\R^2 \setminus \{ y_{n+1} \}$. By Lemma~\ref{QHjana} we obtain
  \begin{eqnarray*}
    k_{S_a}(x_n,y_n) & \le & \log \left(  2 \left( \frac{\sqrt{3}a^n}{4}+ \sqrt{ \left( \frac{a^n}{4} \right)^2 +  \left( \frac{\sqrt{3}a^n}{4} \right)^2 } \right) \right) - \log\left(\frac{a^n}{2}\right)\\
    && + \log \left(  2 \left( \frac{a^n}{4\sqrt{3}}+ \sqrt{ \left( \frac{a^n}{4} \right)^2 +  \left( \frac{a^n}{4\sqrt{3}} \right)^2 } \right) \right) - \log\left(\frac{a^n}{2}\right)\\
    &=&  \log \big(2 + \sqrt{3}\big)  + \log \sqrt{3} = \log\big(3+ 2\sqrt{3} \big). 
  \end{eqnarray*}
Similarly $k_{S_a}(x_n,y_{n+1})$ is estimated by  the quasihyperbolic length of line segments $[x_n,y_{n+1}]$ in $\R^2 \setminus \{ z_{n} \}$and thus by Lemma~\ref{QHjana} we obtain

  \begin{eqnarray*}
    k_{S_a}(x_n,y_{n+1}) & \le & \log \left(  2 \left( \frac{a^n}{2\sqrt{3}}+ \sqrt{ \left( \frac{a^{n+1}}{2} \right)^2 +  \left( \frac{a^n}{2\sqrt{3}} \right)^2 } \right) \right) - \log a^{n+1}\\
    & = & \log \frac{1+\sqrt{3a^2+1}}{\sqrt{3}a}.
  \end{eqnarray*}
  Therefore we have
  \begin{eqnarray}\label{koch:beta1}
    k_{S_a}(x_0,x_n) & \le k_{S_a}(x_0,x_1) + (n-1) \Big(\log\big(3 +2 \sqrt{3}\big)+\log \frac{1+\sqrt{3a^2+1}}{\sqrt{3}a} \Big).
  \end{eqnarray}
  We easily obtain
  \begin{eqnarray*}
    \dist(x_n, \partial S_a) = \sqrt{\left( \frac{a^{n+1}}{2} \right)^2 + \left( \frac{a^n}{2 \sqrt{3}} \right)^2 } = \frac{a^n}{2}\sqrt{a+1/3}
  \end{eqnarray*}
  and thus
  \begin{eqnarray}\label{koch:dist}
    \log \frac{1}{\dist(x_n, \partial S_a)} = \log \frac{2}{\sqrt{a+1/3}} + n \log \frac{1}{a}.
  \end{eqnarray}
  Combining \eqref{koch:beta1} with \eqref{koch:dist} we obtain that $S_a$ satisfies the $\beta$-QHBC for $\beta \le \beta_1$.

  We prove next the lower bound $\beta_2$. We estimate $k(y_n,y_{n+1})$ by the quasihyperbolic distance between $y_n$ and $y_{n+1}$ in the domain $\R^2 \setminus \{ z_n \}$. We deduce that $|y_{n+1}-z_n| = a^{n+1}/2$,
  \begin{eqnarray*}
      |y_n-z_n| & = & \sqrt{\left( \frac{a^n}{2} \right)^2 + \left( \frac{a^n-a^{n+1}}{2} \right)^2 - \frac{a^n}{2}\frac{a^n-a^{n+1}}{2}\cos \frac{\pi}{3}}\\
    & = & \frac{a^n \sqrt{(1+a)(3+2a)}}{2 \sqrt{2}}
  \end{eqnarray*}
  and by sine rule
  \[
    \sin \measuredangle(y_n,z_n,y_{n+1}) = \frac{\sqrt{3}}{\sqrt{2(1+a)(3+2a)}}.
  \]
  Therefore, by Proposition \ref{MOprop}
  \begin{eqnarray*}
    k(y_n,y_{n+1}) & \ge & \sqrt{\arcsin^2 \frac{\sqrt{3}}{\sqrt{2(1+a)(3+2a)}} + \log^2 \frac{ \sqrt{(1+a)(3+2a)}}{a \sqrt{2}}}.
  \end{eqnarray*}
  and
  \begin{equation}\label{koch:beta2}
    k(x_0,x_n) \ge (n-1) \sqrt{\arcsin^2 \frac{\sqrt{3}}{\sqrt{2(1+a)(3+2a)}} + \log^2 \frac{ \sqrt{(1+a)(3+2a)}}{a \sqrt{2}}}.
  \end{equation}
  Combining \eqref{koch:beta2} with \eqref{koch:dist} we obtain that $S_a$ does not satisfy $\beta$-QHBC for $\beta \ge \beta_2$.
\end{proof}

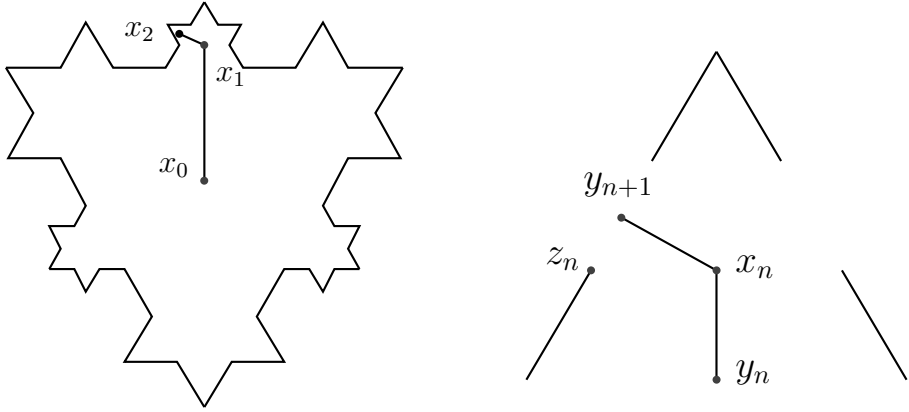
\begin{figure}[ht]
\psset{xunit=3.0cm,yunit=3.0cm,algebraic=true,dotstyle=o,dotsize=3pt 0,linewidth=0.8pt,arrowsize=3pt 2,arrowinset=0.25}
\begin{pspicture*}(-1.04,-1.13)(1.11,0.84)
\psline(-0.87,0.5)(-0.63,0.5)(-0.52,0.7)(-0.4,0.5)(-0.17,0.5)(-0.11,0.6)(-0.16,0.69)(-0.06,0.69)(0,0.79)
\psline(0.87,0.5)(0.63,0.5)(0.52,0.7)(0.4,0.5)(0.17,0.5)(0.11,0.6)(0.16,0.69)(0.06,0.69)(0,0.79)
\psline(0.87,0.5)(0.75,0.3)(0.86,0.1)(0.63,0.1)(0.52,-0.11)(0.57,-0.2)(0.68,-0.2)(0.63,-0.3)(0.68,-0.39)
\psline(0,-1)(0.12,-0.8)(0.35,-0.8)(0.23,-0.6)(0.35,-0.39)(0.46,-0.39)(0.52,-0.49)(0.57,-0.39)(0.68,-0.39)
\psline(0,-1)(-0.12,-0.8)(-0.35,-0.8)(-0.23,-0.6)(-0.35,-0.39)(-0.46,-0.39)(-0.52,-0.49)(-0.57,-0.39)(-0.68,-0.39)
\psline(-0.87,0.5)(-0.75,0.3)(-0.86,0.1)(-0.63,0.1)(-0.52,-0.11)(-0.57,-0.2)(-0.68,-0.2)(-0.63,-0.3)(-0.68,-0.39)
\psline(0,0)(0,0.6)(-0.11,0.65)
\rput[tl](-0.2,0.1){$ x_0 $}
\rput[tl](0.05,0.5){$ x_1 $}
\rput[tl](-0.35,.7){$ x_2 $}
\begin{scriptsize}
\psdots[dotstyle=*,linecolor=darkgray](0,0)
\psdots[dotstyle=*,linecolor=darkgray](0,0.6)
\psdots[dotstyle=*](-0.11,0.65)
\end{scriptsize}
\end{pspicture*}
\psset{xunit=5.0cm,yunit=5.0cm,algebraic=true,dotstyle=o,dotsize=3pt 0,linewidth=0.8pt,arrowsize=3pt 2,arrowinset=0.25}
\begin{pspicture*}(-0.66,-0.15)(0.74,1.01)
\psline(-0.5,0)(-0.33,0.29)
\psline(-0.17,0.58)(0,0.87)
\psline(0.17,0.58)(0,0.87)
\psline(0.5,0)(0.33,0.29)
\psline(0,0)(0,0.29)(-0.25,0.43)
\begin{scriptsize}
\psdots[dotstyle=*,linecolor=darkgray](0,0.29)
\psdots[dotstyle=*,linecolor=darkgray](-0.33,0.29)
\psdots[dotstyle=*,linecolor=darkgray](-0.25,0.43)
\psdots[dotstyle=*,linecolor=darkgray](0,0)
{\large
\rput[tl](0.05,0.33){$ x_n $}
\rput[tl](0.05,0.05){$ y_n $}
\rput[tl](-0.35,0.55){$ y_{n+1} $}
\rput[tl](-0.45,0.35){$ z_n $}
}
\end{scriptsize}
\end{pspicture*}
\caption{Points $x_n$, $y_n$ and $z_n$ as in the proof of Theorem \ref{koch:QHBC}.\label{pic:geodeesi3}}
\end{figure}

\begin{thm}\label{koch:John}
Let $a \in (0,\tfrac12]$. The set $S_a$ is $c$-John with $c= \max\Big\{2, \tfrac{4}{3(1-a)}\Big\}$ and it is not $c'$-John for any $c'<2$. 
\end{thm}

Note that the result is sharp in the range $a \in (0, \tfrac13]$.

\begin{proof}
  Let us denote by $T_0$ the open equilateral triangle, which has sidelength 1 and is contained in $S_a$. We choose the John center $x_0$ to be the center of $T_0$ and let $x$ be any point in $S_a$.

  If $x \in T_0$, then we choose $\gamma$ to be the line segment joining $x$ to $x_0$. It is clear that
  \begin{equation}\label{koch:john1}
    \frac{\dist(\gamma(t), \partial S_a)}{\ell(\gamma(t))} \ge \frac{1}{2}, 
  \end{equation}
(and hence every open equilateral triangle is $2$-John).

  If $x \notin T_0$ then $x \in T_n$, where $T_n$ is a maximal equilateral triangle in $S_a \setminus T_0$ with sidelength $a^n$. Let $s$ be the side of $T_n$ with $s \cap S_a = \emptyset$ and $y_n$ the midpoint of $s$ (see Figure \ref{pic:geodeesi3}). We denote $\gamma = [x,y_n] \cup[y_n,x_n]\cup [x_n,y_{n-1}] \cup \cdots \cup [y_1,x_0]$, where $x_n$ is the center of $T_n$ as in Figure \ref{pic:geodeesi3}. We easily obtain that $|y_{n+1}-x_n| = |x_n-y_n| = a^n/(2\sqrt{3})$ and thus $|y_{n+1}-x_n| + |x_n-y_n| = a^n/\sqrt{3}$. This yields for every $k=0,\ldots n$ that
\[
\ell(\gamma_{y_k}) = \frac{1}{\sqrt 3} ( a^n + \ldots + a^k) = \frac{a^k}{\sqrt 3}\frac{1-a^{n-k+1}}{1-a},
\]
where $\gamma_{y_k}$ is the subpath of $\gamma$ that joins $x$ to $y_k$. Since $\dist (y_k, \partial S_a) = \tfrac{\sqrt3}{4} a^k$ we obtain  for every $a$, $n$ and $k$ that 
  \begin{equation}\label{koch:john2}
    \frac{\dist(\gamma(t), \partial S_a)}{\ell(\gamma(t))} = \frac{\tfrac{\sqrt3}{4} a^k}{\frac{a^k}{\sqrt 3}\frac{1-a^{n-k+1}}{1-a}} = \frac34 \, \frac{1-a}{1-a^{n-k+1}} \ge \frac{3}{4} (1-a), 
  \end{equation}
  where $\gamma(t) = y_k$. Note that the last inequality is sharp when $k$ is fixed and $n \to \infty$. When $a \in [0, \tfrac13]$, we have $\tfrac{3}{4} (1-a) \ge \tfrac12$; the inequality is  sharp when $a =\tfrac13$.
By \eqref{koch:john1} and \eqref{koch:john2} the set $S_a$ is $\max\Big\{2, \tfrac{4}{3(1-a)}\Big\}$-John.

Next we show that $S_a$ is not $c$-John for any $c < 2$. Let us denote by $y$ one of the corners of $T_0$ and consider $\gamma =[z,x_0]$ for $z \in [x_0,y]$. We obtain that $S_a$ is not $c$-John for $c < c_z = 3|x_0-z|$. As $z \to y$ we have $|x_0-y| \to 2/3$ and thus $c_z \to 2$ implying the assertion.
\end{proof}


\vspace{0.5cm}

\noindent\small{\textsc{P.\ Harjulehto and R.\ Kl\'en}}\\
\small{Department of Mathematics and Statistics,
FI-20014 University of Turku, Finland}\\
\footnotesize{\texttt{petteri.harjulehto@utu.fi, riku.klen@utu.fi}}\\

\begin{thebibliography}{9}

\bibitem{BucK95}
S.\ Buckley and  P.\ Koskela:
Sobolev-Poincaré implies John, \textit{Math.\ Res.\ Lett.}~\textbf{2} (1995), no.\ 5, 577--593. 

\bibitem{Fal90}
K. \ Falconer: \textit{Fractal Geometry. Mathematical foundations and Applications.}, John Wiley \& Sons Ltd., Chichester, 1990.

\bibitem{GehM85}
F.\ W.\ Gehring and O.\  Martio:  Lipschitz classes and quasiconformal mappings,
\textit{Ann.\ Acad.\ Sci.\ Fenn.\ Ser.\ A\ I\ Math.}~\textbf{10} (1985), 203--219.

\bibitem{GehP76}
F.\ W.\ Gehring and B.\ P.\ Palka:  Quasiconformally homogeneous domains,
\textit{J.\ Analyse\ Math.}~\textbf{30} (1976), 172--199.

\bibitem{HurMV}
R.\ Hurri-Syrjänen, N.\ Marola and A.\ V.\ Vähäkangas:
Poincar\'e inequalities in quasihyperbolic boundary condition domains, preprint.

\bibitem{Joh61}
F.\ John: Rotation and strain, \textit{Comm.\ Pure Appl. Math.}~\textbf{14} (1961), 391--413.

\bibitem{KelP10}
T.\ Keleti and E.\ Paquette: The trouble with von Koch curves built from $n$-gons, Amer. Math. Monthly 117 (2010), no. 2, 124--137.

\bibitem{Kle09}
R.\ Kl\'en: On hyperbolic type metrics.,
\textit{Ann. \ Acad. \ Sci. \ Fenn. \ Math. \ Diss.} No. \ 152 (2009).

\bibitem{KosR97}
P.\ Koskela and S.\ Rohde: Hausdorff dimension and mean porosity,
\textit{Math.\ Ann.}~\textbf{309} (1997),  no.\ 4, 593--609.

\bibitem{Lap91}
M.\ L.\ Lapidus:  Fractal drum, inverse spectral problems for elliptic operators and a
 partial resolution of the Weyl--Berry conjecture,
\textit{Trans.\ Amer.\ Math.\ Soc.}~\textbf{325} (1991),  no.\ 2, 465--529.

\bibitem{MarOsg86} 
G.\ J.\ Martin and B.\ G.\ Osgood:
The quasihyperbolic metric and the associated estimates on the
hyperbolic metric, \emph{J.\ Anal.\ Math.}~{\bf 47} (1986), 37--53.

\bibitem{MarS79}
O.\ Martio and J.\ Sarvas: Injective theorems in plane and space, \textit{Ann.\ Acad.\ Sci.\ Fenn.\ Ser.\ A\ I\ Math.}~\textbf{4} (1978/1979), 383--401.

\end{thebibliography}
\end{document}